\documentclass[11pt]{article}

\usepackage{amsmath,amssymb,amsthm}
\usepackage{bm}
\usepackage{geometry}
\usepackage{hyperref}
\newtheorem{theorem}{Theorem}
\newtheorem{lemma}{Lemma}
\newtheorem{proposition}{Proposition}
\newtheorem{corollary}{Corollary}

\theoremstyle{remark}
\newtheorem{remark}{Remark}

\title{Functional Bias and Tangent-Space Geometry in Variational Inference}
\author{Sean Plummer \\ snplmmr@gmail.com}
\date{\today}

\begin{document}

\maketitle

\begin{abstract}
Variational inference approximates Bayesian posterior distributions by
projecting onto a tractable family of distributions. While most
theoretical analyses evaluate the quality of this approximation using
global divergence measures, many applications rely on specific posterior
summaries such as expectations, variances, or tail probabilities. We develop a geometric framework for analyzing the bias of posterior
functionals under variational approximations. We show that the leading-order bias of a posterior functional
is determined by its component orthogonal to the variational
tangent space induced by the variational family. Functionals aligned with this space incur only
second-order bias. For structured mean-field variational families we characterize the
tangent space explicitly and show that it consists of block-additive
functions of the parameter blocks, while interaction components
determine the leading-order bias. Under standard local asymptotic
normality conditions we further derive explicit asymptotic expansions
for the bias of posterior functionals and show that omitted interaction
directions produce first-order distortion of cross-block dependence
measures. These results provide a geometric explanation for several well-known
properties of mean-field variational inference, including the
systematic distortion of cross-block dependencies.
\end{abstract}

\section{Introduction}

Variational inference (VI) has become a widely used approach for
approximate Bayesian inference in complex statistical models \cite{blei2017variational, jordan1999variational, wainwright2008graphical}.
By restricting inference to a tractable family of distributions and
optimizing a divergence criterion, variational methods provide scalable
approximations to posterior distributions that are otherwise
computationally intractable.

Despite their computational advantages, variational approximations
introduce systematic bias whose structure is not yet fully understood.
Much of the theoretical analysis of variational inference focuses on
global measures of approximation quality, such as
Kullback--Leibler divergence or posterior contraction rates
\cite{wang2019frequentist, yang2020alpha, zhang2020convergence, alquier2016properties, alquier2020concentration}. While these results provide important guarantees about the overall
behavior of the approximation, many practical applications rely on
specific posterior summaries, including expectations, variances,
covariances, or tail probabilities.
This raises a natural question: \emph{which posterior functionals can
be accurately estimated from a variational approximation?}

In this paper we study this question from a geometric perspective.
Viewing the variational solution as the Kullback--Leibler projection of
the posterior onto a restricted variational family, we show that the leading-order bias of a posterior functional is determined by its component orthogonal to the \textit{variational tangent space}
induced by the variational family. Functionals lying within this space incur only second-order bias in the residual, whereas components orthogonal to it generate leading-order error.
This decomposition provides a structural characterization of the
posterior summaries that can be accurately represented by a given
variational family.

The geometric structure underlying this decomposition is closely related to classical results in semiparametric inference, where efficient
estimators are characterized by orthogonal projections onto tangent
spaces \cite{bickel1993efficient,vanderVaart2000}.
In the variational setting the tangent space is determined by the
variational family itself, and the orthogonal complement identifies
the directions of posterior variation that the variational
approximation fails to capture.
This interpretation leads naturally to a notion of
\emph{variational influence functions} governing the leading-order bias
of posterior functionals.

For structured mean-field variational families, the tangent space
admits a simple explicit characterization \cite{jordan1999variational, wainwright2008graphical}.
In particular, we show that the tangent space consists of
block-additive functions of the parameter blocks,
while its orthogonal complement corresponds to interaction terms
coupling multiple blocks.
This leads to a transparent explanation of several well-known
properties of mean-field approximations:
additive summaries of parameter blocks are accurately represented,
whereas summaries sensitive to cross-block dependence exhibit
leading-order bias.

We further investigate the asymptotic implications of this geometric
structure in regular parametric models \cite{vanderVaart2000, kleijn2012bernstein}.
When the posterior distribution admits a local Gaussian approximation,
variational approximations behave as Gaussian projections with
restricted covariance structure.
This yields explicit asymptotic expansions for the bias of posterior
functionals.
In particular, interaction-sensitive functionals such as
cross-covariances exhibit first-order asymptotic distortion under
mean-field approximations, whereas functionals aligned with the
variational tangent space automatically eliminate the leading bias
term.

Taken together, these results show that the geometry of the variational
family determines not only which posterior summaries are accurately
represented, but also how the bias of variational approximations
behaves in large samples.

\paragraph{Contributions.}
The main contributions of this paper are as follows.

\begin{enumerate}
\item We derive a functional bias decomposition for variational
Kullback--Leibler projections that expresses the bias of a posterior
functional in terms of the orthogonal complement of the variational
tangent space.

\item We show that posterior functionals aligned with the tangent space
exhibit only second-order bias under variational approximation.

\item For structured mean-field families, we characterize the tangent
space explicitly and show that it consists of block-additive functions,
with interaction components determining the leading-order bias.

\item We illustrate the framework through examples involving posterior
covariance, variance of linear contrasts, and joint tail probabilities.

\item Under local asymptotic normality conditions we derive explicit
asymptotic bias expansions and show that omitted interaction directions
produce first-order distortion of cross-block dependence measures.
\end{enumerate}

The remainder of the paper is organized as follows.
Section~\ref{sec:projection} introduces the variational projection
framework and establishes a key orthogonality property of the
variational tangent space.
Section~\ref{sec:functional_bias} derives the functional bias
decomposition for variational approximations.
Section~\ref{sec:struct_mf} characterizes the tangent space for
structured mean-field families, and Section~\ref{sec:two_block}
develops an explicit two-block decomposition.
Section~\ref{sec:examples} presents examples illustrating the
implications of the theory.
Section~\ref{sec:asymptotics} studies the local asymptotic behavior of
the bias decomposition, and Section~\ref{sec:discussion} concludes.

\section{Variational Projection Framework} \label{sec:projection}

Let $(\Theta,\mathcal{A},\pi_0)$ be a measurable space equipped with a
reference probability measure $\pi_0$. We assume that both the posterior
distribution and the variational family are absolutely continuous with
respect to $\pi_0$.

\subsection{Posterior and variational family}

Let the posterior distribution be
\[
\pi(d\theta) = p(\theta)\,\pi_0(d\theta),
\]
where $p(\theta)>0$ denotes the posterior density with respect to
$\pi_0$. Consider a parametric variational family
\[
\mathcal{Q}
=
\{q_\lambda(d\theta) = r_\lambda(\theta)\,\pi_0(d\theta)
:\lambda\in\Lambda\subset\mathbb{R}^k\}.
\]
Variational inference selects an approximation by minimizing the
Kullback--Leibler divergence \cite{blei2017variational,wainwright2008graphical}
\[
F(\lambda)
=
\mathrm{KL}(q_\lambda\|\pi)
=
\int r_\lambda(\theta)
\log \left [ \frac{r_\lambda(\theta)}{p(\theta)} \right ]
\,\pi_0(d\theta).
\]
Assume that a minimizer $\lambda^*\in\Lambda$ exists and define
$q^* = q_{\lambda^*}$. Throughout the paper we interpret $q^*$ as the
\emph{KL projection} of the posterior distribution onto the variational
family. Expectations with respect to $q^*$ will be taken in the Hilbert space
$L^2(q^*) = L^2(\Theta,\mathcal A,q^*)$.

\subsection{Residual representation}

A convenient quantity for comparing $q^*$ and $\pi$ is the log-density
residual
\[
\Delta(\theta)
=
\log\frac{q^*(\theta)}{\pi(\theta)}.
\]
The residual $\Delta$ measures the log-density discrepancy between the
variational approximation and the posterior. Note that $\Delta(\theta)$ is positive where the variational
approximation overestimates the posterior density and negative
where it underestimates it. Equivalently,
\[
\frac{\pi(\theta)}{q^*(\theta)}
=
e^{-\Delta(\theta)}.
\]
This representation yields the identity
\[
E_\pi[g]
=
E_{q^*}\!\left[g\,e^{-\Delta}\right]
\]
for any $g\in L^1(\pi)$. This relationship will allow us to relate
expectations under $\pi$ and $q^*$. 

\subsection{Tangent space}

The geometry of the variational family can be described using its
tangent space at $q^*$. Let
\[
s_j(\theta)
=
\partial_{\lambda_j}\log r_\lambda(\theta)\big|_{\lambda=\lambda^*},
\qquad j=1,\dots,k,
\]
denote the score functions of the variational family. The tangent space of the variational family at $q^*$ is defined as
\[
T_{q^*}\mathcal Q
=
\overline{\mathrm{span}\{s_1,\dots,s_k\}}^{L^2(q^*)}.
\]
Intuitively, $T_{q^*}\mathcal Q$ describes the directions in which the
variational distribution can be locally perturbed while remaining within
the variational family.

\subsection{Score orthogonality}

The KL optimality condition for $\lambda^*$ implies that the residual
$\Delta$ is orthogonal to the tangent space.

\begin{lemma}[Tangent-space orthogonality]
\label{lem:tangent_orthogonal}
For all $h \in T_{q^*}\mathcal Q$,
\[
E_{q^*}[h\,\Delta] = 0 .
\]
\end{lemma}

\begin{proof}
First note that differentiation of the KL objective yields
\[
\partial_{\lambda_j}F(\lambda^*)
=
E_{q^*}[s_j(\Delta+1)].
\]
Since $\lambda^*$ is an interior minimizer,
$\partial_{\lambda_j}F(\lambda^*)=0$.
Normalization of $q_\lambda$ implies
$E_{q^*}[s_j]=0$, and therefore
\[
E_{q^*}[s_j\,\Delta]=0
\]
for each $j=1,\dots,k$. Since the score functions span $T_{q^*}(\mathcal Q)$ in
$L^2(q^*)$, the result follows by linearity and continuity of the inner product.
\end{proof}

Lemma~\ref{lem:tangent_orthogonal} shows that the residual $\Delta$ lies
in the orthogonal complement of the variational tangent space. This
geometric property will drive the bias decomposition developed in the
next section.


\section{Functional Bias Decomposition}\label{sec:functional_bias}

We now relate posterior expectations to expectations under the
variational approximation. The key observation is that the residual
\[
\Delta(\theta)=\log\frac{q^*(\theta)}{\pi(\theta)}
\]
provides a convenient change-of-measure representation.

\paragraph{Assumption.}
There exists $\eta>0$ such that
\[
E_{q^*}[e^{\eta |\Delta|}] < \infty .
\]
This condition holds whenever the density ratio $\pi/q*$ has subexponential tails.
\subsection{Change-of-measure expansion}

The following lemma expresses the difference between posterior and
variational expectations in terms of the residual $\Delta$.

\begin{lemma}[Change-of-measure expansion]
\label{lem:change_of_measure}
Let $g\in L^2(q^*)$ and define $\rho(x) = e^{-x} - 1 + x$. Then
\[
E_\pi[g] - E_{q^*}[g]
=
- E_{q^*}[g\Delta]
+
E_{q^*}[g\,\rho(\Delta)] .
\]
Moreover,
\[
\left|E_{q^*}[g\,\rho(\Delta)]\right|
\le
C
\|g - E_{q^*}[g]\|_{L^2(q^*)}
\|\Delta\|_{L^2(q^*)}^2
\]
for a constant $C$ depending only on the integrability properties of
$\Delta$.
\end{lemma}

\begin{proof}
Since $\pi(\theta)/q^*(\theta) = e^{-\Delta(\theta)}$, we have
\[
E_\pi[g] = E_{q^*}[g\,e^{-\Delta}] .
\]
Using the identity $e^{-x}=1-x+\rho(x)$ yields
\[
E_\pi[g]-E_{q^*}[g]
=
- E_{q^*}[g\Delta]
+
E_{q^*}[g\rho(\Delta)] .
\]
The remainder bound follows from the inequality
\[
|\rho(x)| \le \tfrac12 x^2 e^{|x|},
\]
which follows from the Taylor remainder formula, and the assumed exponential integrability of $\Delta$.
\end{proof}

\subsection{Geometric bias decomposition}

We now combine the change-of-measure expansion with the tangent-space
orthogonality property established in Section~\ref{sec:projection}.

\begin{theorem}[Variational projection identity]
\label{thm:variational_bias}
Let $q^*$ denote the KL projection of $\pi$ onto a variational family
$\mathcal Q$. Let $T_{q^*}\mathcal Q$ denote the tangent space of
$\mathcal Q$ at $q^*$ and $\Delta = \log(q^*/\pi)$. For any fixed $g\in L^2(q^*)$,  write the orthogonal decomposition
\[
g = g_\parallel + g_\perp \quad \text{with} \quad g_\parallel \in T_{q^*}\mathcal Q,
\qquad
g_\perp \perp T_{q^*}\mathcal Q .
\]
Then
\[
E_\pi[g] - E_{q^*}[g]
=
- \langle g_\perp , \Delta \rangle_{L^2(q^*)}
+
O\!\left(\|\Delta\|_{L^2(q^*)}^2\right)
\]
with implicit constant depending on $g$. 
\end{theorem}

\begin{proof}
From Lemma~\ref{lem:change_of_measure} we obtain
\[
E_\pi[g] - E_{q^*}[g]
=
- E_{q^*}[g\Delta]
+
E_{q^*}[g\rho(\Delta)] .
\]
Write the orthogonal decomposition
\[
g = g_\parallel + g_\perp,
\]
where $g_\parallel \in T_{q^*}\mathcal Q$ and $g_\perp \perp T_{q^*}\mathcal Q$. By Lemma~\ref{lem:tangent_orthogonal}, the residual $\Delta$ is orthogonal to the tangent space,
hence
\[
E_{q^*}[g_\parallel \Delta] = 0.
\]
Therefore
\[
E_{q^*}[g\Delta] = E_{q^*}[g_\perp \Delta].
\]
Substituting this identity into the expansion from Lemma~\ref{lem:change_of_measure} yields
\[
E_\pi[g] - E_{q^*}[g]
=
- E_{q^*}[g_\perp \Delta]
+
E_{q^*}[g\rho(\Delta)].
\]
The remainder term is bounded by
\[
|E_{q^*}[g\rho(\Delta)]|
\le
C\|g - E_{q^*}[g]\|_{L^2(q^*)}\|\Delta\|_{L^2(q^*)}^2
\]
from Lemma~\ref{lem:change_of_measure}, which yields the stated result.
\end{proof}

\begin{remark}[Projection identity]
Theorem 1 shows that the leading-order bias of a posterior functional
is determined by the component of the functional orthogonal to the
variational tangent space. In particular,

\[
E_\pi[g] - E_{q^*}[g]
=
- \langle (I-P_T)g, \Delta \rangle_{L^2(q^*)}
+ O(\|\Delta\|^2_{L^2(q^*)}),
\]
with implicit constant depending on $g$. Above $P_T$ denotes the projection into $T_{q^*}\mathcal{Q}$. Thus only the component of the functional outside the tangent space
contributes to the first-order bias. This quantity plays the role of a variational influence component,
analogous to the influence functions appearing in classical semiparametric theory.
\end{remark}

\begin{corollary}[Unbiased functional class]
For fixed $g \in T_{q^*}\mathcal Q$,
\[
E_\pi[g] - E_{q^*}[g]
=
O\!\left(\|\Delta\|_{L^2(q^*)}^2\right).
\]
\end{corollary}

\subsection{Convex Functionals}
For convex functionals, the same geometric mechanism extends through
subgradients. Let $T$ be a convex functional defined on a convex subset of $L^2(q^*)$ containing $\pi$ and
$q^*$. Recall that
$\psi\in\partial T(\pi)$ means
\[
T(\nu)-T(\pi)\ge \langle \psi,\nu-\pi\rangle
\,\text{ for all }\nu \text{ in the domain of } T,
\]
where $\langle \psi,\nu-\pi\rangle$ denotes the natural dual pairing.
Thus, whenever $T$ admits a first-order expansion at $\pi$ represented
by a subgradient $\psi$, the leading variation of $T$ is governed by
the linear functional $\nu\mapsto \langle \psi,\nu-\pi\rangle$. If this
subgradient lies in the variational tangent space, then the same
orthogonality argument as in Theorem~\ref{thm:variational_bias} removes the first-order term.

We continue with the convention that $\pi$ and $q^*$ are identified with their densities relative to the common dominating measure $\pi_0$.

\begin{corollary}[Convex functionals with subgradient representation]
Let $T$ be a convex functional on a convex subset
of $L^2(q^*)$ containing $\pi$ and $q^*$. Suppose there exists
$\psi\in\partial T(\pi)$ such that $\psi\in T_{q^*}\mathcal Q\cap
L^2(q^*)$, and suppose that $T$ admits a first-order expansion at $\pi$
of the form
\[
T(q^*)-T(\pi)=\langle \psi,q^*-\pi\rangle_{L^2(q^*)} + r(q^*,\pi),
\]
with
\[
r(q^*,\pi)=O\!\left(\|\Delta\|_{L^2(q^*)}^2\right).
\]
Then
\[
T(q^*)-T(\pi)=O\!\left(\|\Delta\|_{L^2(q^*)}^2\right).
\]
\end{corollary} 

\begin{proof}
By assumption,
\[
T(q^*)-T(\pi)=\langle \psi,q^*-\pi\rangle_{L^2(q^*)} + r(q^*,\pi).
\]
The pairing $\langle \psi,q^*-\pi\rangle_{L^2(q^*)}$ is the bias of the linear
functional generated by $\psi$, so applying Theorem~\ref{thm:variational_bias} with $g=\psi$
gives
\[
\langle \psi,q^*-\pi\rangle_{L^2(q^*)}
=
E_\pi[\psi]-E_{q^*}[\psi]
=
-\langle (I-P_T)\psi,\Delta\rangle_{L^2(q^*)}
+O\!\left(\|\Delta\|_{L^2(q^*)}^2\right).
\]
Since $\psi\in T_{q^*}\mathcal Q$, we have $(I-P_T)\psi=0$, and
therefore
\[
\langle \psi,q^*-\pi\rangle_{L^2(q^*)}
=
O\!\left(\|\Delta\|_{L^2(q^*)}^2\right).
\]
Combining this with the assumed bound on $r(q^*,\pi)$ yields
\[
T(q^*)-T(\pi)=O\!\left(\|\Delta\|_{L^2(q^*)}^2\right).
\]
\end{proof}

\section{Geometry of Structured Mean-Field} \label{sec:struct_mf}

We now specialize the previous analysis to structured mean-field
variational families \cite{jordan1999variational,blei2017variational}. In this setting the variational tangent space
admits an explicit characterization, which leads to a clear description
of the posterior functionals that can be accurately represented.

\subsection{Structured mean-field family}

Let the parameter vector be partitioned into blocks $\theta = (\theta_{B_1},\ldots,\theta_{B_m})$. The structured mean-field variational family consists of distributions
of the form
\[
q(\theta) = \prod_{b=1}^m q_b(\theta_{B_b}).
\]
Let
\[
q^*(\theta) = \prod_{b=1}^m q_b^*(\theta_{B_b})
\]
denote the KL projection of the posterior onto this family.

\subsection{Tangent directions}

In Section \ref{sec:projection}, the tangent space was defined for a finite-dimensional
parametric variational family through score functions associated with the
parameter \(\lambda\). For the structured mean-field family, it is more
natural to work directly with the full product class
\[
\mathcal Q_{MF}=\Bigl\{ q(\theta)=\prod_{b=1}^m q_b(\theta_{B_b}) \Bigr\},
\]
which is typically infinite-dimensional. Accordingly, in this section we
define the tangent space at \(q^*\) as the closed linear span in
\(L^2(q^*)\) of score directions generated by smooth blockwise
perturbation paths through \(q^*\).

Perturbations of the mean-field family arise from perturbations of the
individual block factors. Consider a perturbation of block $b$ of the
form
\[
q_{b,\varepsilon}(\theta_{B_b})
=
q_b^*(\theta_{B_b})
\frac{e^{\varepsilon f_b(\theta_{B_b})}}
{E_{q_b^*}[e^{\varepsilon f_b}]},
\]
where $f_b \in L^2(q_{B_b}^*)$. The corresponding perturbation of the joint distribution is
\[
q_\varepsilon(\theta)
=
q_{b,\varepsilon}(\theta_{B_b})
\prod_{j\neq b} q_j^*(\theta_{B_j}).
\]
Differentiating the log density yields
\[
\frac{d}{d\varepsilon}\log q_\varepsilon(\theta)\Big|_{\varepsilon=0}
=
f_b(\theta_{B_b}) - E_{q_b^*}[f_b].
\]
Thus tangent directions generated by perturbations of block $b$
depend only on the variables $\theta_{B_b}$. 

\subsection{Tangent space characterization}

We now characterize the tangent space of the structured mean-field
family.

\begin{theorem}[Tangent space for structured mean-field]
\label{thm:block}
Let \(T_{q^*}\mathcal Q_{MF}\) denote the tangent space at \(q^*\)
generated by smooth blockwise perturbation paths of the structured
mean-field family. Then
\[
T_{q^*}\mathcal Q_{MF}
=
\left\{
\sum_{b=1}^m f_b(\theta_{B_b}) : f_b\in L_0^2(q_{B_b}^*)
\right\},
\]
where $L^2_0(q_{B_b}^*)$ denotes the space of square-integrable
functions with zero mean under $q_{B_b}^*$.
\end{theorem}

\begin{proof}
From the perturbation argument above, any tangent direction generated by
a variation of block $b$ has the form
\[
f_b(\theta_{B_b}) - E_{q_b^*}[f_b].
\]
Therefore the tangent space contains all centered functions depending
on individual blocks. Linear combinations of such perturbations yield
functions of the form
\[
\sum_{b=1}^m f_b(\theta_{B_b})
\]
with each $f_b$ centered under $q_{B_b}^*$. 

Conversely, by construction, any smooth perturbation path within the
structured mean-field family acts through the block factors
\(q_b(\theta_{B_b})\) separately. Therefore its score direction is a sum
of centered functions depending on individual blocks only, and cannot
contain interaction terms involving multiple blocks.
\end{proof}

\subsection{Orthogonal complement}

The orthogonal complement of the tangent space consists of functions
that have zero conditional expectation given each block. In particular,
a function $g \in L^2(q^*)$ lies in $T_{q^*}\mathcal Q_{MF}^\perp$ if
\[
E_{q^*}[g \mid \theta_{B_b}] = 0
\]
for all blocks $b=1,\dots,m$. Because \(q^*\) factorizes across the blocks, orthogonality to all
block-additive functions is equivalent to vanishing conditional
expectation given each block. Functions in this space represent interaction components coupling
multiple parameter blocks.

\subsection{Interpretation}

Theorem~\ref{thm:block} shows that structured mean-field approximations
accurately represent additive functionals of the parameter blocks.
Bias arises from components of a functional that depend on interactions
between blocks. Combining this result with the bias decomposition of
Section~\ref{sec:functional_bias} implies that interaction terms
determine the leading-order bias of variational approximations.


\section{Two-Block Decomposition}\label{sec:two_block}

To illustrate the implications of the previous results, we consider the
simplest nontrivial structured mean-field setting with two parameter
blocks.

\subsection{Two-block mean-field}

Let the parameter vector be partitioned as $\theta = (X,Y)$ and consider
the mean-field variational family $q(X,Y) = q_X(X)\,q_Y(Y)$. Let $q^*(X,Y) = q_X^*(X)\,q_Y^*(Y)$ denote the KL projection of the posterior distribution onto this family. By Theorem~\ref{thm:block}, the variational tangent space
consists of additive functions of the form $f(X) + g(Y)$.

\subsection{Functional decomposition}

Any function $h \in L^2(q^*)$ admits the two-way Hoeffding
(ANOVA) decomposition \cite{hoeffding1948}
\begin{equation}
    h(X,Y)= \mu + h_X(X) + h_Y(Y) + h_{XY}(X,Y), \label{eq:anova}
\end{equation}
where
\[
\mu = E_{q^*}[h],\quad h_X(X) = E_{q^*}[h \mid X] - \mu,\quad h_Y(Y) = E_{q^*}[h \mid Y] - \mu,
\]
and the interaction component satisfies
\[
E_{q^*}[h_{XY} \mid X] = 0,
\qquad
E_{q^*}[h_{XY} \mid Y] = 0.
\]
The functions $h_X$ and $h_Y$ lie in the variational tangent space,
while $h_{XY}$ lies in its orthogonal complement. The decomposition above corresponds to the orthogonal decomposition of $L^2(q^*)$ into additive and interaction components.

\subsection{Bias decomposition}

Combining this representation with the bias decomposition of
Section~\ref{sec:functional_bias} yields the following result.

\begin{proposition}
Let $h \in L^2(q^*)$ admit the decomposition in equation \eqref{eq:anova}. Then
\[
E_\pi[h] - E_{q^*}[h]
=
- E_{q^*}[h_{XY}\Delta]
+
O\!\left(\|\Delta\|_{L^2(q^*)}^2\right),
\]
with implicit constant depending on $h$.
\end{proposition}

Thus the leading-order bias depends only on the interaction component
$h_{XY}$.

\subsection{Interpretation}

The result shows that additive summaries of the parameter blocks are
captured accurately by the mean-field approximation, while bias arises
from interaction terms coupling multiple blocks. This provides a
geometric explanation for several well-known properties of mean-field
variational inference, including the systematic distortion of
cross-block dependencies.

\section{Examples} \label{sec:examples}

We illustrate the implications of the preceding results for several
posterior summaries commonly used in practice.

\subsection{Cross-covariance}

Consider two parameter blocks $(X,Y)$ under a mean-field approximation
$q^*(x,y)=q_X^*(x)q_Y^*(y)$. Let $h(X,Y)=u(X)v(Y)$ for square-integrable functions $u$ and $v$. Using the decomposition of Section~\ref{sec:two_block},
\[
h(X,Y)
=
\mu + h_X(X) + h_Y(Y) + h_{XY}(X,Y),
\]
where
\[
h_{XY}(X,Y)
=
\big(u(X)-E_{q^*}[u(X)]\big)
\big(v(Y)-E_{q^*}[v(Y)]\big).
\]
The functions $h_X$ and $h_Y$ lie in the variational tangent space
$T_{q^*}\mathcal Q$, while $h_{XY}$ lies in its orthogonal complement. Applying Theorem~\ref{thm:variational_bias} yields
\[
E_\pi[u(X)v(Y)] - E_{q^*}[u(X)v(Y)]
=
- E_{q^*}[h_{XY}\Delta]
+
O\!\left(\|\Delta\|_{L^2(q^*)}^2\right).
\]
Thus the leading bias arises entirely from the interaction component.
In particular, cross-block covariance is sensitive to the interaction
structure of the posterior distribution.

\subsection{Variance of a linear functional}

Let $\theta=(\theta_1,\ldots,\theta_d)$ and consider the linear
contrast $L(\theta)=a^\top \theta$. The associated second-moment functional is $h(\theta)=L(\theta)^2$. Expanding yields
\[
h(\theta)
=
\sum_{i=1}^d a_i^2 \theta_i^2
+
2\sum_{i<j} a_i a_j \theta_i\theta_j.
\]
Under a fully factorized mean-field approximation, the first sum
consists of functions depending on individual coordinates and therefore
lies in the tangent space $T_{q^*}\mathcal Q$.  
The second sum contains interaction terms coupling different
coordinates and therefore lies in the orthogonal complement. Consequently, the leading bias in estimating $\mathrm{Var}_\pi(L)$
arises from these interaction components.

\subsection{Joint tail probability}

Consider the functional $F(q)=q\{X>t,\,Y>s\}$. Equivalently, this corresponds to the expectation $h(X,Y)=\mathbf{1}\{X>t,\,Y>s\}$. The projection of this indicator function onto the tangent space
corresponds to additive marginal components of the exceedance event, while the orthogonal
component represents the excess probability of joint exceedance beyond
what is predicted by the marginals. Applying Theorem~\ref{thm:variational_bias} shows that the leading bias
of $F(q^*)$ depends on this interaction component.  
This explains why mean-field approximations may accurately represent
marginal tail probabilities while distorting joint tail behavior.

\subsection{Summary}

These examples illustrate a common pattern: posterior summaries whose
influence functions lie in the tangent space $T_{q^*}\mathcal Q$ are
captured accurately by the variational approximation, while summaries
depending on components orthogonal to the tangent space exhibit
first-order bias.

\section{Local Asymptotic Bias}
\label{sec:asymptotics}

The geometric decomposition developed in Sections~\ref{sec:functional_bias}--\ref{sec:two_block} characterizes the
finite-sample bias of posterior functionals in terms of the component of
the functional lying in the orthogonal complement of the variational
tangent space. In many statistical settings the posterior distribution
depends on a sample of size $n$ and concentrates around the true
parameter value as $n$ grows. It is therefore natural to ask how the
geometric bias decomposition interacts with classical asymptotic
approximations to the posterior distribution.

In regular parametric models the posterior distribution admits a
local Gaussian approximation under standard Bernstein–von Mises
conditions \cite{vanderVaart2000}. Variational
approximations applied to such posteriors therefore behave like
variational approximations to the corresponding local Gaussian limit
experiment. The results of this section analyze the resulting asymptotic
bias of posterior functionals and show how the geometric structure
identified earlier determines the leading asymptotic error.

Let $\pi_n$ denote the posterior distribution based on $n$ observations,
and let $q_n^*$ denote the variational approximation obtained by
minimizing $\mathrm{KL}(q\|\pi_n)$ over $\mathcal Q$.
Define the residual
\[
\Delta_n(\theta)
=
\log\frac{q_n^*(\theta)}{\pi_n(\theta)} .
\]
The bias decomposition of Section~\ref{sec:functional_bias} applies to
each pair $(\pi_n,q_n^*)$.

\begin{proposition}[Asymptotic projection identity] \label{prop:asymp}
Let $g \in L^2(q_n^*)$. Write
\[
g = g_{\parallel,n} + g_{\perp,n},
\qquad
g_{\parallel,n} \in T_{q_n^*}\mathcal Q,
\quad
g_{\perp,n} \perp T_{q_n^*}\mathcal Q .
\]
Then
\[
E_{\pi_n}[g] - E_{q_n^*}[g]
=
- E_{q_n^*}[g_{\perp,n}\Delta_n]
+
O\!\left(\|\Delta_n\|_{L^2(q_n^*)}^2\right),
\]
with implicit constant depending on $g$.
\end{proposition}

The projection identity of Proposition~\ref{prop:asymp} shows that the asymptotic bias
continues to be governed by the interaction between the residual
$\Delta_n$ and the component of the functional orthogonal to the
variational tangent space. To obtain explicit bias expressions we now
consider the local Gaussian regime that arises in regular parametric
models. In this regime both the posterior distribution and the
variational approximation concentrate at rate $n^{-1/2}$ and can be
approximated by Gaussian distributions with covariance matrices scaled
by $1/n$. This allows the leading bias term to be computed explicitly.

\begin{theorem}[Local asymptotic bias expansion] \label{thm:local_bias}
Suppose the posterior and variational approximation satisfy
\[
\pi_n \overset{\mathrm{TV}}{\to} N(\mu_n,\Sigma/n),
\qquad
q_n^* \overset{\mathrm{TV}}{\to} N(\mu_n,V/n)
\]
for fixed positive definite matrices $\Sigma$ and $V$. Let $g:\mathbb R^d\to\mathbb R$ be three times continuously
differentiable in a neighborhood of $\mu_n$ with bounded third
derivatives. Then
\[
E_{\pi_n}[g(\theta)] - E_{q_n^*}[g(\theta)]
=
\frac{1}{2n}
\mathrm{tr}\!\big(H_g(\mu_n)(\Sigma-V)\big)
+
o_p(n^{-1}).
\]
\end{theorem}

\begin{proof}[Proof sketch]
Write $\theta = \mu_n + Z/\sqrt n$. Under $\pi_n$, $Z\sim N(0,\Sigma)$, while under $q_n^*$,
$Z\sim N(0,V)$ up to $o_p(1)$ error. A Taylor expansion of $g$ around $\mu_n$ gives
\[
g(\theta)
=
g(\mu_n)
+
\frac{1}{\sqrt n}\nabla g(\mu_n)^\top Z
+
\frac{1}{2n}Z^\top H_g(\mu_n)Z
+
R_n(Z)
\]
with $E[R_n(Z)]=o(n^{-1})$. The linear terms vanish after expectation, and the quadratic terms yield
\[
E_{\pi_n}[Z^\top H_g Z] - E_{q_n^*}[Z^\top H_g Z]
=
\mathrm{tr}\!\big(H_g(\Sigma-V)\big).
\]
The result follows.
\end{proof}

\begin{corollary}[Mean-field interaction bias]
If $q_n^*$ is a mean-field approximation so that $V$ is diagonal,
then the leading bias term depends on the interaction-sensitive
components of the Hessian $H_g(\mu_n)$.
\end{corollary}

\begin{corollary}
For $g(\theta)=\theta_i\theta_j$ with $i\neq j$,
\[
E_{\pi_n}[\theta_i\theta_j] - E_{q_n^*}[\theta_i\theta_j]
=
\frac{\Sigma_{ij}}{n}
+
o(n^{-1}).
\]
\end{corollary}

The expansion in Theorem~\ref{thm:local_bias} shows that the leading asymptotic bias is
determined by the mismatch between the covariance structure of the
posterior distribution and the covariance structure representable by
the variational family. For mean-field approximations this mismatch
occurs precisely in the interaction directions that lie outside the
variational tangent space. The next result shows that this observation
aligns exactly with the geometric theory developed earlier: functionals
whose influence directions lie in the tangent space automatically
eliminate the leading asymptotic bias term.

\begin{theorem}[Bias reduction for tangent-space functionals]
Suppose the conditions of Theorem~\ref{thm:local_bias} hold.
Suppose $q_n^*$ is mean-field and let $g$ be a block-additive functional satisfying $g \in T_{q_n^*}\mathcal Q $. Then
\[
E_{\pi_n}[g] - E_{q_n^*}[g]
=
o_p(n^{-1}).
\]
\end{theorem}

 \begin{proof}
From Theorem~\ref{thm:local_bias},
\[
E_{\pi_n}[g] - E_{q_n^*}[g]
=
\frac{1}{2n}
\operatorname{tr}(H_g(\Sigma-V))
+
o_p(n^{-1}).
\]
If $g \in T_{q_n^*}\mathcal Q$, then $g$ can be expressed locally as a
sum of functions depending on individual variational blocks.
Consequently $H_g$ contains no cross-block interaction terms. For mean-field approximations the matrix $V$ is diagonal, so the
interaction components of $\Sigma-V$ vanish when contracted with $H_g$.
Therefore
\[
\operatorname{tr}(H_g(\Sigma-V)) = 0,
\]
which yields the result.
\end{proof}

A particularly transparent consequence of Theorem~\ref{thm:local_bias} appears when
considering cross-covariance functionals under mean-field
approximations.

\begin{proposition}[Asymptotic distortion of cross-covariances under mean-field]
\label{prop:cross_cov_bias}
Suppose the conditions of Theorem~\ref{thm:local_bias} hold, and assume
that $q_n^*$ is a mean-field Gaussian approximation so that its local
covariance matrix $V$ is diagonal. Let
\[
g(\theta)=\theta_i\theta_j,
\qquad i\neq j.
\]
Then
\[
E_{\pi_n}[\theta_i\theta_j]-E_{q_n^*}[\theta_i\theta_j]
=
\frac{\Sigma_{ij}}{n}
+
o_p(n^{-1}).
\]
In particular, if $\Sigma_{ij}\neq 0$, then the cross-covariance
functional exhibits nonvanishing first-order asymptotic bias under the
mean-field approximation.
\end{proposition}

\begin{proof}
For $g(\theta)=\theta_i\theta_j$ with $i\neq j$, the Hessian is
\[
H_g = e_i e_j^\top + e_j e_i^\top,
\]
where $e_i$ denotes the $i$th standard basis vector. Applying
Theorem~\ref{thm:local_bias} gives
\[
E_{\pi_n}[g(\theta)]-E_{q_n^*}[g(\theta)]
=
\frac{1}{2n}\operatorname{tr}\!\big(H_g(\Sigma-V)\big)
+
o_p(n^{-1}).
\]
Since $V$ is diagonal, $(\Sigma-V)_{ij}=\Sigma_{ij}$ for $i\neq j$. A
direct calculation shows
\[
\operatorname{tr}\!\big(H_g(\Sigma-V)\big)=2(\Sigma-V)_{ij}=2\Sigma_{ij},
\]
and therefore
\[
E_{\pi_n}[\theta_i\theta_j]-E_{q_n^*}[\theta_i\theta_j]
=
\frac{\Sigma_{ij}}{n}
+
o_p(n^{-1}).
\]
\end{proof}

Proposition~\ref{prop:cross_cov_bias} shows that, in the local Gaussian
regime, omitted interaction directions under mean-field variational
inference translate directly into first-order asymptotic distortion of
cross-covariance functionals.


\section{Discussion} \label{sec:discussion}

We have developed a geometric framework for understanding the bias of
posterior functionals under variational approximations. The analysis
shows that bias is governed by the orthogonal complement of the
variational tangent space $T_{q^*}\mathcal Q$. Functionals aligned with
this space incur only second-order bias, while functionals containing
components orthogonal to it exhibit first-order error. This projection identity is analogous to classical results in semiparametric theory, where estimation error depends on the component of the influence function orthogonal to the model tangent space \cite{bickel1993efficient,vanderVaart2000}.

For structured mean-field families, the tangent space admits an
explicit characterization as the set of block-additive functions. This
leads to a simple interpretation: additive summaries of the parameter
blocks are accurately represented by the variational approximation,
while bias arises from interaction components coupling multiple blocks.
The two-block decomposition in Section~\ref{sec:two_block} illustrates
how this interaction structure drives bias in common posterior
summaries.

\subsection{Implications for variational inference}

The results provide a geometric explanation for several well-known
properties of mean-field variational inference. In particular, posterior
summaries involving cross-block interactions, such as covariances or
joint tail probabilities, depend on directions orthogonal to the
variational tangent space and therefore exhibit leading-order bias.
Conversely, additive summaries of individual parameter blocks lie in
the tangent space and tend to be accurately represented.

This perspective complements existing theoretical analyses of
variational inference that focus on global divergence measures or
posterior contraction \cite{wang2019frequentist, yang2020alpha, zhang2020convergence, alquier2020concentration}. While such results quantify the overall
approximation error, the present framework identifies which posterior
functionals are most sensitive to the structural constraints of the
variational family.

\subsection{Structured variational families}

The characterization of the tangent space also clarifies why structured
variational families often improve over fully factorized
approximations. Enlarging the block structure expands the tangent space
and therefore reduces the dimension of its orthogonal complement.
Consequently, fewer posterior functionals exhibit leading-order bias.
This observation provides a geometric interpretation of the empirical
success of structured mean-field methods.

\subsection{Future directions}

Several extensions of the present framework may be of interest. First,
the analysis could be applied to richer variational families, including
mixture approximations or normalizing-flow-based methods, in order to
understand how their geometric structure influences functional bias.
Second, the relationship between the variational tangent space and the
convergence properties of coordinate ascent variational inference
algorithms merits further investigation. Finally, it may be useful to
develop diagnostic tools that identify posterior summaries lying near
the orthogonal complement of the tangent space, as these are the
quantities most susceptible to variational bias.

\medskip

Overall, the results suggest that evaluating variational approximations
through the lens of posterior functionals provides a useful complement
to traditional divergence-based assessments. These results suggest that variational families should be evaluated not only by global divergence measures but also by the function classes represented by their tangent spaces.

\section*{AI Acknowledgment}
The author used a large language model to assist with editing, drafting, and organizational refinement of the manuscript. All mathematical arguments, proofs, experimental implementations, and interpretations are the author’s own work and were independently verified.

\bibliographystyle{plain}
\bibliography{refs}

\end{document}